\documentclass[preprint,11pt,a4paper]{elsarticle}

\usepackage{geometry}
\usepackage[T1]{fontenc}
\usepackage{lmodern}

\usepackage{amsmath,amssymb,amsthm,mathtools}
\usepackage{mathrsfs}
\usepackage{microtype}
\usepackage{enumitem}
\usepackage{xcolor}

\usepackage{titlesec}
\titleformat*{\subsection}{\normalfont\normalsize\bfseries}

\biboptions{numbers,sort&compress}

\usepackage[
colorlinks=true,
linkcolor=blue,
citecolor=blue,
urlcolor=blue
]{hyperref}

\newtheorem{theorem}{Theorem}[section]
\newtheorem{lemma}[theorem]{Lemma}
\newtheorem{proposition}[theorem]{Proposition}

\newtheorem*{hurwitztheorem}{Hurwitz theorem}
\newtheorem*{gausslucastheorem}{Gauss--Lucas theorem}

\theoremstyle{definition}

\theoremstyle{remark}

\numberwithin{equation}{section}

\newcommand{\Rea}{\operatorname{Re}}

\allowdisplaybreaks[2]
\date{}
\begin{document}
	
	\begin{frontmatter}
		
		\title{Weak Hurwitz stability of extended Tur\'an expressions
			for Stirling--Whitney--Riordan and Laguerre polynomials}
		
		\author{Wenle Shi}
		\ead{shi-wenle@hotmail.com}
		
		\address{School of Mathematical Sciences,
			Dalian University of Technology,
			Dalian 116024, P. R. China}
		
		\begin{abstract}
			In this paper, we establish a sufficient condition for weak
			Hurwitz stability of extended Tur\'an expressions associated
			with polynomial sequences satisfying a first-order
			differential recurrence.
			As applications, we prove weak Hurwitz stability of extended
			Tur\'an expressions for the Stirling--Whitney--Riordan
			polynomials, several families of Eulerian polynomials,
			the Bell polynomials, and the reflected generalized
			Laguerre polynomials.
			These results settle two conjectures of Chasse, Grabarek,
			and Visontai concerning the Bell and Laguerre polynomials.
		\end{abstract}
		
		\begin{keyword}
			weak Hurwitz stability
			\sep Laguerre polynomial
			\sep Stirling--Whitney--Riordan polynomial
			\sep Tur\'an expression
			\MSC[2020] 30C15 \sep 33C45 \sep 05A15
		\end{keyword}
		
	\end{frontmatter}

\section{Introduction}

For a polynomial sequence $\mathcal P=(P_j(x))_{j\geq0}$,
its \emph{$m$th extended Tur\'an expression}
\cite[Eq.~(14)]{CGV} is given by
\begin{align}\label{eq:turan}
	\mathscr T_{m,n}(\mathcal P;x)
	=\sum_{j=0}^{2n}\frac{(-1)^{n+j}}{j!(2n-j)!}
	P_{m+j}(x)P_{m+2n-j}(x),
	\qquad m\geq0,\ n\geq1.
\end{align}
For $n=1$, we obtain the ordinary Tur\'an expression
$$\mathscr T_{m,1}(\mathcal P;x)
=P_{m+1}(x)^2-P_m(x)P_{m+2}(x).$$
A nonzero polynomial $f(x)$ is \emph{weakly Hurwitz stable} if it has
no zeros in $\Rea x>0$.

The Hurwitz stability of polynomials has been widely studied.
Borcea and Br\"and\'en \cite{BBstability} characterized linear
operators preserving polynomial stability.
Chasse, Grabarek, and Visontai \cite{CGV} studied
the stability of Tur\'an expressions for some
orthogonal polynomials. 
Ding and Zhu \cite{DZhurwitz} established criteria for Hurwitz
stability of polynomials satisfying differential recurrences
and for weak Hurwitz stability of their Tur\'an expressions.
Liu and Yan \cite{LiuYan} obtained weak Hurwitz stability results
for reverse strong Tur\'an expressions of several combinatorial
polynomial sequences.
Subsequently, Gao and Liu \cite{GaoLiu} proved weak Hurwitz stability
of strong Tur\'an expressions for Eulerian and Narayana polynomials. 
Ding and Zhu~\cite{DZ24} established a real stability
criterion for polynomials satisfying differential recurrences.
Zhu \cite{ZSWR} introduced the Stirling--Whitney--Riordan 
polynomials and established conditions for weak Hurwitz stability
of their Tur\'an expressions. 

As a common generalization of the Stirling--Whitney--Riordan, Eulerian, and reflected
Laguerre polynomials, we define $\mathcal P=(P_j(x))_{j\geq0}$
by $P_0(x)=1$ and
\begin{align}\label{eq:unified-recurrence}
	(1+\sigma j)P_{j+1}(x)
	&=\bigl(a_2+\rho j
	+(b_1+b_2+\tau a_1j)(x+\lambda)
	+\tau b_1j(x+\lambda)^2\bigr)P_j(x)\notag\\
	&\quad+\bigl(a_1(x+\lambda)
	+(b_1-\tau a_1)(x+\lambda)^2
	-\tau b_1(x+\lambda)^3\bigr)D_xP_j(x),
\end{align}
where $D_x=\mathrm d/\mathrm dx$ and
$a_1,a_2,b_1,b_2,\lambda,\tau,\sigma,\rho\geq0$.
Our main result is the following.

\begin{theorem}\label{thm:unified}
	Let $\mathcal P=(P_j(x))_{j\geq0}$ be defined by
	\eqref{eq:unified-recurrence} with $P_0(x)=1$.
	Suppose that $a_1,a_2,b_1,b_2,\lambda,\tau,\sigma,\rho\geq0$ and
	\begin{align}\label{eq:unified-parameters}
		a_1(b_1+b_2)\geq a_2b_1,\qquad
		b_1+b_2+\tau a_2>0,\qquad a_1+b_1>0.
	\end{align}
	Assume that one of the following conditions holds:
	\begin{enumerate}
		\item[\rm(i)] $\sigma=\rho=0$;
		\item[\rm(ii)] $\sigma>0$, $b_1=\tau=0$, $\rho=a_1>0$, and
		$a_2/a_1-1/\sigma\in\mathbb Z_{\geq0}$.
	\end{enumerate}
	Then $\mathscr T_{m,n}(\mathcal P;x)$ is nonzero and has no zeros
	in $\Rea x>-\lambda$ for every $m\geq0$ and $n\geq1$.
	In particular, it is weakly Hurwitz stable.
\end{theorem}

Table~\ref{tab:families} lists polynomial sequences satisfying
the conditions of Theorem~\ref{thm:unified} and
the recurrence \eqref{eq:unified-recurrence}.
For further details on these polynomials, see
\cite{MMYY,Szego,Simanek,GL23,ZSWR}.

\begin{table}[htbp]
	\caption{Special cases of Theorem~\ref{thm:unified}.}
	\label{tab:families}
	\centering
	\small
	\setlength{\tabcolsep}{3pt}
	\renewcommand{\arraystretch}{1.12}
	\begin{tabular}{@{}lcl@{}}
		\hline
		Polynomial
		& $(a_1,a_2,b_1,b_2,\lambda,\tau,\sigma,\rho)$
		& Parameter range\\
		\hline
		Stirling--Whitney--Riordan polynomial $S_j(x)$
		& $(a_1,a_2,b_1,b_2,\lambda,0,0,0)$
		& \eqref{eq:unified-parameters}\\
		
		Eulerian polynomial $A_j(x)$
		& $(1,1,0,0,0,1,0,0)$
		& ---\\
		
		Type $B$ Eulerian polynomial $A^B_j(x)$
		& $(2,1,0,1,0,1,0,0)$
		& ---\\
		
		$\ell$-colored Eulerian polynomial $A_j^{[\ell]}(x)$
		& $(\ell,1,0,\ell-1,0,1,0,0)$
		& $\ell\in\mathbb Z_{\geq1}$\\
		
		
		Cycle-weighted Eulerian polynomial $A_j(x;\theta)$
		& $(1,\theta,0,0,0,1,0,0)$
		& $\theta>0$\\
		
		Type $A$ $1/k$-Eulerian polynomial $A_j^{(k)}(x)$
		& $(k,1,0,0,0,1,0,0)$
		& $k>0$\\
		
		Type $B$ $1/k$-Eulerian polynomial $B_j^{(k)}(x)$
		& $(2k,1,0,1,0,1,0,0)$
		& $k>0$\\
		
		Bell polynomial $B_j(x)$
		& $(1,0,0,1,0,0,0,0)$
		& ---\\
		
		$r$-Bell polynomial $B_{j,r}(x)$
		& $(1,r,0,1,0,0,0,0)$
		& $r\in\mathbb Z_{\geq0}$\\
		
		Geometric polynomial $G_j(x)$
		& $(1,0,1,0,0,0,0,0)$
		& ---\\
		
		Dowling polynomial $D_j^{(r)}(x)$
		& $(r,1,0,1,0,0,0,0)$
		& $r>0$\\
		
		Tanny--geometric polynomial $F_j^{(r)}(x)$
		& $(r,1,1,0,0,0,0,0)$
		& $r\geq1$\\
		
		Riordan polynomial $R_j(x)$
		& $(1,0,0,1,1,0,0,0)$
		& ---\\
		
		\shortstack[l]{Reflected generalized Laguerre polynomial\\
			$L_j^{(N)}(-x-\lambda)$}
		& $(1,N+1,0,1,\lambda,0,1,1)$
		& $N\in\mathbb Z_{\geq0}$, $\lambda\geq0$\\
		\hline
	\end{tabular}
\end{table}

It is well known that the Bell polynomials are given by
\begin{align*}
	B_j(x)=\sum_{r=0}^{j}
	\left\{\begin{matrix}j\\r\end{matrix}\right\}x^r,
\end{align*}
where $\left\{\begin{smallmatrix}j\\r\end{smallmatrix}\right\}$
is the Stirling number of the second kind.
For $\alpha>-1$, the generalized Laguerre polynomials are defined by
\begin{align*}
	L_j^{(\alpha)}(x)
	=\sum_{r=0}^{j}\binom{j+\alpha}{j-r}\frac{(-x)^r}{r!}.
\end{align*}
In particular, the Laguerre polynomials are given by
$L_j(x)=L_j^{(0)}(x)$ for $j\geq0$. 
Chasse, Grabarek, and Visontai
\cite{CGV} proposed the following
two conjectures concerning the Bell and Laguerre polynomials.

\begin{theorem}[{\cite{CGV}, Conjecture~4.3}]\label{thm:bell}
	Let $\mathcal B=(B_j(x))_{j\geq0}$ be the sequence of
	Bell polynomials. For all integers $m\geq0$ and $n\geq1$,
	the extended Tur\'an expression
	$\mathscr T_{m,n}(\mathcal B;x)$ is weakly Hurwitz stable.
\end{theorem}

\begin{theorem}[{\cite{CGV}, Conjecture~4.4}]\label{thm:lag}
	Let $\mathcal L=(L_j(x))_{j\geq0}$ be the sequence of
	Laguerre polynomials. For all integers $m\geq0$ and $n\geq1$, 
	the extended Tur\'an expression
	$\mathscr T_{m,n}(\mathcal L;-x)$ is weakly Hurwitz stable.
\end{theorem}

Both results follow directly from Theorem~\ref{thm:unified}.
In part~{\rm(i)}, taking $a_1=b_2=1$ and
$a_2=b_1=\lambda=\tau=\sigma=\rho=0$ gives
$P_j(x)=B_j(x)$, proving Theorem~\ref{thm:bell}.
Similarly, in part~{\rm(ii)}, taking
$a_1=a_2=b_2=\sigma=\rho=1$ and $b_1=\lambda=\tau=0$
gives $P_j(x)=L_j(-x)$, proving Theorem~\ref{thm:lag}.

We use the following forms of Hurwitz theorem
\cite[p.~4]{Marden} and the Gauss--Lucas theorem
\cite[Theorem~(6,1), p.~22]{Marden}.

\begin{hurwitztheorem}
	Let $f_k$ be holomorphic and nonvanishing on a connected
	open set $\Omega\subseteq\mathbb C$.
	If $f_k\to f$ uniformly on compact subsets of $\Omega$,
	then $f$ is either identically zero or nonvanishing on $\Omega$.
\end{hurwitztheorem}

\begin{gausslucastheorem}
	Every zero of the derivative of a nonconstant complex polynomial
	lies in the convex hull of its zeros.
	The same holds for each nonconstant higher derivative.
\end{gausslucastheorem}

The paper is organized as follows.
Section~\ref{sec:proof-i} combines interval and unit disk lemmas
with a fractional linear transformation to prove
Theorem~\ref{thm:unified}~{\rm(i)}.
Section~\ref{sec:proof-ii} uses polynomial approximations
and Hurwitz's theorem to prove
Theorem~\ref{thm:unified}~{\rm(ii)}.

\section{Proof of Theorem~\ref{thm:unified}~(i)}\label{sec:proof-i}

\subsection{Interval and disk lemmas}

Let $[z^k]p(z)$ denote the coefficient of $z^k$ in $p(z)$.
The next result gives a polynomial transformation that preserves
the location of zeros in $[-1,0]$.

\begin{lemma}\label{lem:interval}
	Let $p(z)$ be a nonzero real polynomial of degree $k\geq0$
	with all its zeros in $[-1,0]$. For $\ell>0$, define
	\begin{align}
		P_\ell(z)
		&=4z(1+z)D_z^{2}p(z)
		+[4+(4\ell+6)z]D_zp(z)
		+2\ell p(z).
		\label{eq:Pd}
	\end{align}
	Then $P_\ell(z)$ has degree $k$, and all its zeros lie
	in $[-1,0]$.
\end{lemma}

\begin{proof}
Without loss of generality, we may assume that the leading
coefficient of $p(z)$ is positive.
By~\eqref{eq:Pd}, we have $[z^k]P_\ell(z)>0$.
	Since the right-hand side of \eqref{eq:Pd} is a polynomial
	of degree at most $k$, it follows that $P_\ell(z)$ has
	degree $k$. By Bates and Yoshida
	\cite[Theorem~2.7]{BY}, all zeros of $P_\ell(z)$ are real.
	
	It remains to exclude zeros in $(-\infty,-1)\cup(0,\infty)$.
	Since $p(z)$ has positive leading coefficient and all its
	zeros lie in $[-1,0]$, we have $D_z^{j}p(z)>0$ for $z>0$ and $0\leq j\leq k$.
	Thus, for $z>0$, the right-hand side of \eqref{eq:Pd}
	is positive, and hence $P_\ell(z)>0$.
	
For $z<-1$, we have
\begin{align*}
	(-1)^{k-j}D_z^jp(z)>0,
	\qquad 0\leq j\leq k.
\end{align*}
Moreover,
\begin{align*}
	4z(1+z)>0,\qquad
	4+(4\ell+6)z<0,\qquad
	2\ell>0.
\end{align*}
Therefore every nonzero summand in \eqref{eq:Pd} has
sign $(-1)^k$. Since $2\ell p(z)\neq0$, it follows that
$(-1)^kP_\ell(z)>0$ for $z<-1$. 
Consequently, $P_\ell(z)$ has no zeros outside $[-1,0]$,
which completes the proof.
\end{proof}

We  use the following more general form of the
Gauss--Lucas theorem, which we state for convenience.
\begin{lemma}[{\cite[Theorem~(6,1)$'$, p.~22]{Marden}}]
	\label{lem:convex}
	Let $r_1,\ldots,r_k\in\mathbb C$, and suppose that
	$\rho_1,\ldots,\rho_k$ are nonnegative real numbers satisfying
	$\sum_{j=1}^{k}\rho_j>0$.
	If $z\in\mathbb C\setminus\{r_1,\ldots,r_k\}$ satisfies
	\begin{align*}
		\sum_{j=1}^{k}\frac{\rho_j}{z-r_j}=0,
	\end{align*}
	then $z$ belongs to the convex hull of
	$\{r_j:\rho_j>0\}$.
\end{lemma}

We write $\partial_x=\partial/\partial x$.
We next give a result on the location of zeros in the unit disk.

\begin{lemma}\label{lem:coefficient}
	Suppose that $d>0$, $n,\ell\in\mathbb Z_{\geq0}$, and
	$|\eta_j|\leq1$ for $1\leq j\leq\ell$.
	Then the polynomial
	\begin{align}\label{eq:G}
		G_{n,d,\ell}(z)
		=[t^{2n}]
		\frac{\bigl(1-(\tanh(t))^2\bigr)^d
			\prod_{j=1}^{\ell}
			\bigl(1-\eta_j(\tanh(t))^2\bigr)}
		{\bigl(1-z(\tanh(t))^2\bigr)^{d+\ell}}
	\end{align}
	has degree $n$ and all its zeros in $|z|\leq1$.
\end{lemma}

\begin{proof}
We use induction on $\ell$.
	We first consider the case $\ell=0$. For $d>0$, define
	\begin{align*}
		E_d(t,z)
		&:=\bigl(1-z(\sinh(t))^2\bigr)^{-d},
		\notag\\
		V_{d,n}(z)
		&:=(2n)![t^{2n}]E_d(t,z).
	\end{align*}
	
	By differentiating with respect to $z$ and $t$, we have
	\begin{align}
		\partial_z E_d(t,z)
		&=\frac{d(\sinh(t))^2}
		{\bigl(1-z(\sinh(t))^2\bigr)^{d+1}},\quad
		\partial_z^2 E_d(t,z)
		=\frac{d(d+1)(\sinh(t))^4}
		{\bigl(1-z(\sinh(t))^2\bigr)^{d+2}},
		\notag\\
		\partial_t^2 E_d(t,z)
		&=\frac{2dz\bigl(1+2(\sinh(t))^2\bigr)}
		{\bigl(1-z(\sinh(t))^2\bigr)^{d+1}}+
		\frac{4d(d+1)z^2(\sinh(t))^2(\cosh(t))^2}
		{\bigl(1-z(\sinh(t))^2\bigr)^{d+2}}.
		\label{eq:Ederivatives}
	\end{align}
	We observe that
	\begin{align*}
		\frac{1}{\bigl(1-z(\sinh(t))^2\bigr)^{d+1}}
		&=E_d(t,z)
		+\frac{z}{d}\partial_z E_d(t,z),
		\\
		\frac{(\sinh(t))^2}
		{\bigl(1-z(\sinh(t))^2\bigr)^{d+2}}
		&=\frac{1}{d}\partial_z E_d(t,z)
		+\frac{z}{d(d+1)}
		\partial_z^2 E_d(t,z).
	\end{align*}
By \eqref{eq:Ederivatives} and
$(\cosh(t))^2=1+(\sinh(t))^2$, we obtain
\begin{align}\label{eq:EPDE}
	\partial_t^2 E_d(t,z)
	&=\frac{2dz(1+2(\sinh(t))^2)}
	{\bigl(1-z(\sinh(t))^2\bigr)^{d+1}}+
	\frac{4d(d+1)z^2(\sinh(t))^2(\cosh(t))^2}
	{\bigl(1-z(\sinh(t))^2\bigr)^{d+2}}\notag\\
	&=4z^2(1+z)\partial_z^2 E_d(t,z)+z[4+(4d+6)z]\partial_z E_d(t,z)
	+2dzE_d(t,z).
\end{align}
	
	Since $E_d(t,z)$ is even in $t$, its Taylor expansion is
	\begin{align*}
		E_d(t,z)
		=\sum_{n=0}^{\infty}V_{d,n}(z)\frac{t^{2n}}{(2n)!}.
	\end{align*}
	Differentiating this series twice with respect to $t$,
	we obtain
	\begin{align*}
		\partial_t^2 E_d(t,z)
		=\sum_{n=1}^{\infty}
		V_{d,n}(z)\frac{(2n)(2n-1)t^{2n-2}}{(2n)!}
		=\sum_{n=0}^{\infty}
		V_{d,n+1}(z)\frac{t^{2n}}{(2n)!}.
	\end{align*}
	Similarly, differentiating once and twice with respect
	to $z$, respectively, we have
	\begin{align*}
		\partial_z E_d(t,z)
		=\sum_{n=0}^{\infty}
		D_zV_{d,n}(z)
		\frac{t^{2n}}{(2n)!},
		\quad
		\partial_z^2 E_d(t,z)
		=\sum_{n=0}^{\infty}
		D_z^{2}V_{d,n}(z)
		\frac{t^{2n}}{(2n)!}.
	\end{align*}
	Substituting these expansions into \eqref{eq:EPDE}
	and comparing the coefficients of $t^{2n}/(2n)!$
	on both sides, we get
	\begin{align}\label{eq:Vrecurrence}
		V_{d,n+1}(z)
		=z\biggl[
		4z(1+z)D_z^{2}V_{d,n}(z)+[4+(4d+6)z]D_zV_{d,n}(z)
		+2dV_{d,n}(z)
		\biggr],
	\end{align}
	where $V_{d,0}(z)=1$.
	
We claim that $V_{d,n}(z)$ has degree $n$ and all its zeros
in $[-1,0]$ for every $n\geq0$.
We prove the claim by induction on $n$.
For $n=0$, the polynomial $V_{d,0}(z)=1$ has degree $0$
and no zeros.
Suppose that $V_{d,n}(z)$ has degree $n$ and all its zeros
in $[-1,0]$ for some $n\geq0$.
Lemma~\ref{lem:interval} and \eqref{eq:Vrecurrence}
then imply that $V_{d,n+1}(z)$ has degree $n+1$
and all its zeros in $[-1,0]$.
This completes the induction on $n$ and proves the claim.
	
	By the definition of $G_{n,d,0}(z)$, we have 
	\begin{align*}
		G_{n,d,0}(z)
		=[t^{2n}]
		\frac{\bigl(1-(\tanh(t))^2\bigr)^d}
		{\bigl(1-z(\tanh(t))^2\bigr)^d}
		=[t^{2n}]\bigl(1-(z-1)(\sinh(t))^2\bigr)^{-d}
		=\frac{V_{d,n}(z-1)}{(2n)!}.
	\end{align*}
	Hence $G_{n,d,0}(z)$ has degree $n$ and all its zeros
	in $[0,1]$, which proves the case $\ell=0$.
	
	Suppose that, for some $\ell\geq0$, $G_{n,d,\ell}(z)$
	has degree $n$ and all its zeros in the closed unit disk.
	Fix $\eta_{\ell+1}$ with $|\eta_{\ell+1}|\leq1$.
	Differentiating \eqref{eq:G} with respect to $z$,
	with $\eta_1,\ldots,\eta_\ell$ held fixed, we have
	\begin{align}\label{eq:Gz-diff}
		D_zG_{n,d,\ell}(z)
		=(d+\ell)[t^{2n}]
		\frac{(\tanh(t))^2
			\bigl(1-(\tanh(t))^2\bigr)^d
			\prod_{j=1}^{\ell}
			\bigl(1-\eta_j(\tanh(t))^2\bigr)}
		{\bigl(1-z(\tanh(t))^2\bigr)^{d+\ell+1}}.
	\end{align}
	By~\eqref{eq:G} and~\eqref{eq:Gz-diff},
	we obtain
	\begin{align}
		G_{n,d,\ell+1}(z)
		&=[t^{2n}]
		\frac{\bigl(1-(\tanh(t))^2\bigr)^d
			\prod_{j=1}^{\ell}
			\bigl(1-\eta_j(\tanh(t))^2\bigr)}
		{\bigl(1-z(\tanh(t))^2\bigr)^{d+\ell}}
		\frac{1-\eta_{\ell+1}(\tanh(t))^2}
		{1-z(\tanh(t))^2}
		\notag\\
		&=[t^{2n}]
		\frac{\bigl(1-(\tanh(t))^2\bigr)^d
			\prod_{j=1}^{\ell}
			\bigl(1-\eta_j(\tanh(t))^2\bigr)}
		{\bigl(1-z(\tanh(t))^2\bigr)^{d+\ell}}
		\left(
		1+\frac{(z-\eta_{\ell+1})(\tanh(t))^2}
		{1-z(\tanh(t))^2}
		\right)
		\notag\\
		&=G_{n,d,\ell}(z)
		+\frac{z-\eta_{\ell+1}}{d+\ell}
		D_zG_{n,d,\ell}(z).
		\label{eq:Ginduction}
	\end{align}
	
	By the induction hypothesis, $G_{n,d,\ell}(z)$ has degree $n$.
	Thus the right-hand side of \eqref{eq:Ginduction}
	has degree at most $n$.
	Since $d+\ell>0$, we have
	\begin{align*}
		[z^n]G_{n,d,\ell+1}(z)
		=\left(1+\frac{n}{d+\ell}\right)
		[z^n]G_{n,d,\ell}(z)\neq0.
	\end{align*}
	Hence $G_{n,d,\ell+1}(z)$ has degree $n$.
	
	Let $r_1,\ldots,r_n$ be the zeros of $G_{n,d,\ell}(z)$,
	counted with multiplicity.
	By the induction hypothesis, $|r_j|\leq1$ for every $j$.
	Suppose that $G_{n,d,\ell+1}(z_0)=0$ for some $|z_0|>1$.
	It follows from~\eqref{eq:Ginduction} that 
	\begin{align*}
		0
		=G_{n,d,\ell}(z_0)
		+\frac{z_0-\eta_{\ell+1}}{d+\ell}
		\left.
		D_zG_{n,d,\ell}(z)
		\right|_{z=z_0}.
	\end{align*}
	Since $G_{n,d,\ell}(z_0)\neq0$ and
	$z_0\neq\eta_{\ell+1}$, we may divide both sides
	of this equality by
	$(z_0-\eta_{\ell+1})G_{n,d,\ell}(z_0)/(d+\ell)$.
	We obtain
	\begin{align*}
		0
		=\frac{d+\ell}{z_0-\eta_{\ell+1}}
		+\frac{1}{G_{n,d,\ell}(z_0)}
		\left.
		D_zG_{n,d,\ell}(z)
		\right|_{z=z_0}
		=\frac{d+\ell}{z_0-\eta_{\ell+1}}
		+\sum_{j=1}^{n}\frac{1}{z_0-r_j}.
	\end{align*}
	Since $d+\ell>0$,
	Lemma~\ref{lem:convex} implies that $z_0$ belongs to the
	convex hull of $\eta_{\ell+1},r_1,\ldots,r_n$.
	All these points lie in the closed unit disk.
	Since this disk is convex, we have $|z_0|\leq1$,
	a contradiction.
	Therefore all zeros of $G_{n,d,\ell+1}(z)$ lie in the
	closed unit disk, which completes the induction.
\end{proof}

\subsection{A polynomial recurrence}

For $a>0$, write $(a)_0=1$ and
$(a)_j=a(a+1)\cdots(a+j-1)$ for $j\geq1$. The next result shows that a differential recurrence generates
polynomials with all their zeros in $[-1,1]$.

\begin{lemma}\label{lem:R}
	Assume that $d>0$ and $-d\leq h\leq d$, and define
	\begin{align}\label{eq:Rrecurrence}
		R_{j+1}(u)
		=(u^2-1)D_uR_j(u)
		+(du+h)R_j(u),\qquad R_0(u)=1.
	\end{align}
	Then $R_j(u)$ has degree $j$, leading coefficient $(d)_j$,
	and all its zeros in $[-1,1]$.
\end{lemma}

\begin{proof}
	We proceed by induction on $j$. Since $R_0(u)=1$, the conclusion holds for $j=0$. 
	Suppose that it holds for $R_j(u)$. Comparing leading coefficients in
	\eqref{eq:Rrecurrence}, we obtain
	\begin{align*}
		[u^{j+1}]R_{j+1}(u)=(j+d)(d)_j=(d)_{j+1}>0.
	\end{align*}
	Thus $R_{j+1}$ has degree $j+1$.
	Let $r_1,\ldots,r_j\in[-1,1]$ be the zeros of $R_j$,
	counted with multiplicity. Suppose that $R_{j+1}(u_0)=0$ for $u_0\notin[-1,1]$.
	By the induction hypothesis, $R_j(u_0)\neq0$.
	Also, $u_0^2-1\neq0$.
	Evaluating \eqref{eq:Rrecurrence} at $u=u_0$ and dividing
	both sides by $(u_0^2-1)R_j(u_0)$ gives
	\begin{align}\label{eq:Rroots}
		0&=\frac{\left.D_uR_j(u)\right|_{u=u_0}}{R_j(u_0)}
		+\frac{du_0+h}{u_0^2-1}\notag\\
		&=\sum_{i=1}^{j}\frac{1}{u_0-r_i}
		+\frac{(d+h)/2}{u_0-1}
		+\frac{(d-h)/2}{u_0+1}.
	\end{align}
	Since $d>0$ and $|h|\leq d$, Lemma~\ref{lem:convex}
	applies to \eqref{eq:Rroots} and implies that
	$u_0\in[-1,1]$, a contradiction.
	This completes the induction.
\end{proof}

For $d>0$, $-d\leq h\leq d$, fixed $u\in\mathbb C$,
and $t$ sufficiently close to $0$, define
\begin{align}
	\phi(t,u)&=\frac{u-\tanh(t)}{1-u\tanh(t)},\label{eq:phi}\\
	H(t,u)&=\frac{(1+\tanh(t))^{(d+h)/2}
		(1-\tanh(t))^{(d-h)/2}}
	{(1-u\tanh(t))^d}.\label{eq:phiH}
\end{align}
For each fixed $u\in\mathbb C$, the powers in $H(t,u)$
are defined by analytic branches near $t=0$, chosen so that
each power equals $1$ at $t=0$.
Thus $H(t,u)$ is analytic in $t$ near $0$ and $H(0,u)=1$.

\begin{lemma}
	Let $\mathcal R=(R_j(u))_{j\geq0}$ be defined by
	\eqref{eq:Rrecurrence} with $R_0(u)=1$, where
	$d>0$ and $-d\leq h\leq d$.
	With $\phi(t,u)$ and $H(t,u)$ defined by
	\eqref{eq:phi} and \eqref{eq:phiH}, respectively, the identity
	\begin{align}\label{eq:flowseries}
		H(t,u)R_m(\phi(t,u))
		=\sum_{j=0}^{\infty}R_{m+j}(u)\frac{t^j}{j!}
	\end{align}
	holds for each integer $m\geq0$ and fixed $u\in\mathbb C$,
	provided that $t$ is sufficiently close to $0$.
\end{lemma}

\begin{proof}
	Fix $u\in\mathbb C$.
	Since $\tanh(t)$ is analytic near $t=0$ and vanishes there,
	the definitions in \eqref{eq:phi} and \eqref{eq:phiH}
	give functions analytic near $t=0$, with
	$H(0,u)=1$ and $\phi(0,u)=u$.
	Differentiating with respect to $t$, we obtain
	\begin{align}
		\partial_t \phi(t,u)
		&=\frac{(u^2-1)(1-(\tanh(t))^2)}
		{(1-u\tanh(t))^2}\notag\\
		&=\phi(t,u)^2-1 
		\label{eq:phiderivative}
			\end{align}
			and
			\begin{align}
		\frac{\partial_t H(t,u)}{H(t,u)}
		&=\frac{d+h}{2}(1-\tanh(t))
		-\frac{d-h}{2}(1+\tanh (t))
		+\frac{du(1-(\tanh(t))^2)}{1-u\tanh(t)}\notag\\
		&=h-d\tanh(t)
		+\frac{du(1-(\tanh(t))^2)}{1-u\tanh(t)}\notag\\
		&=d\phi(t,u)+h.
		\label{eq:Hderivative}
	\end{align}
By the chain rule, we obtain
\begin{align}\label{eq:shift}
	\partial_t\bigl(H(t,u)R_j(\phi(t,u))\bigr)
	&=\partial_tH(t,u)\,R_j(\phi(t,u))+H(t,u)
	\left.D_vR_j(v)\right|_{v=\phi(t,u)}
	\partial_t\phi(t,u)\notag\\
	&=H(t,u)\Bigl[
	\bigl(\phi(t,u)^2-1\bigr)
	\left.D_vR_j(v)\right|_{v=\phi(t,u)}
	+(d\phi(t,u)+h)R_j(\phi(t,u))
	\Bigr]\notag\\
	&=H(t,u)R_{j+1}(\phi(t,u)).
\end{align}
Here the second equality follows from
\eqref{eq:phiderivative} and \eqref{eq:Hderivative},
and the last follows from \eqref{eq:Rrecurrence}.
	By repeated application of \eqref{eq:shift}, we obtain
	\begin{align*}
		\left.
		\partial_t^r
		\bigl(H(t,u)R_m(\phi(t,u))\bigr)
		\right|_{t=0}
		=H(0,u)R_{m+r}(\phi(0,u))
		=R_{m+r}(u),
		\qquad r\geq0.
	\end{align*}
	Taylor's formula therefore gives \eqref{eq:flowseries}.
\end{proof}

\begin{proposition}\label{lem:RTuran}
	Assume that $d>0$ and $-d\leq h\leq d$.
	Let $\mathcal{R}=(R_j(u))_{j\geq0}$ be defined by
	\eqref{eq:Rrecurrence}.
	Then $\mathscr{T}_{m,n}(\mathcal{R};u)\neq0$
	for all integers $m\geq0$ and $n\geq1$ and all
	$u\in\mathbb C$ with $|u|>1$.
\end{proposition}

\begin{proof}
	Fix $u\in\mathbb C$ with $|u|>1$.
	By Lemma~\ref{lem:R}, the degree of $R_m(v)$ is $m$,
	and all its zeros lie in $[-1,1]$.
	Let $r_1,\ldots,r_m$ be the zeros of $R_m(v)$,
	counted with multiplicity.
	Then
	\begin{align*}
		R_m(v)=C_m\prod_{j=1}^{m}(v-r_j),
	\end{align*}
	where $C_m\neq0$ is the leading coefficient of $R_m(v)$.
	Since $|u|>1$, we have $u-r_j\neq0$ for every $j$
	and $R_m(u)\neq0$.
	Set
	\begin{align}
		\eta_j=\left(\frac{1-r_ju}{u-r_j}\right)^2,
		\qquad 1\leq j\leq m.
		\label{eq:eta}
	\end{align}
	Since $r_j$ is real and belongs to $[-1,1]$, we have
	$|u-r_j|^2-|1-r_ju|^2=(1-r_j^2)(|u|^2-1)\geq0.$
	Therefore
	\begin{align*}
		|\eta_j|
		=\frac{|1-r_ju|^2}{|u-r_j|^2}\leq1.
	\end{align*}
	
	Using \eqref{eq:phi}, we get
	\begin{align*}
		\phi(t,u)-r_j
		=\frac{(u-r_j)-(1-r_ju)\tanh(t)}
		{1-u\tanh(t)}.
	\end{align*}
	Since $\tanh(-t)=-\tanh(t)$, it follows that
	\begin{align*}
		\frac{(\phi(t,u)-r_j)(\phi(-t,u)-r_j)}
		{(u-r_j)^2}
		&=\frac{(u-r_j)^2-(1-r_ju)^2(\tanh(t))^2}
		{(u-r_j)^2\bigl(1-u^2(\tanh(t))^2\bigr)}\\
		&=\frac{1-\eta_j(\tanh(t))^2}
		{1-u^2(\tanh(t))^2}.
	\end{align*}
	
	By the factorization of $R_m(v)$,
	we have
	\begin{align}\label{eq:Rpair}
		\frac{R_m(\phi(t,u))R_m(\phi(-t,u))}
		{R_m(u)^2}
		&=\prod_{j=1}^{m}
		\frac{(\phi(t,u)-r_j)(\phi(-t,u)-r_j)}
		{(u-r_j)^2}\notag\\
		&=\frac{\prod_{j=1}^{m}
			\bigl(1-\eta_j(\tanh(t))^2\bigr)}
		{\bigl(1-u^2(\tanh(t))^2\bigr)^m}.
	\end{align}
		It follows from \eqref{eq:phiH} that
	\begin{align}
		H(t,u)H(-t,u)
		&=\frac{(1+\tanh(t))^{d}
			(1-\tanh(t))^{d}}
		{(1-u\tanh(t))^{d}
			(1+u\tanh(t))^{d}}\notag\\
		&=\frac{\bigl(1-(\tanh(t))^2\bigr)^{d}}
		{\bigl(1-u^2(\tanh(t))^2\bigr)^{d}}.
		\label{eq:Hpair}
	\end{align}
	
	Multiplying both sides of \eqref{eq:Rpair}
	by $H(t,u)H(-t,u)R_m(u)^2$
	and using \eqref{eq:Hpair}, we obtain
	\begin{align}\label{eq:pairedkernel}
		&H(t,u)H(-t,u)R_m(\phi(t,u))R_m(\phi(-t,u))
		\notag\\
		=&R_m(u)^2
		\frac{\bigl(1-(\tanh(t))^2\bigr)^{d}
			\prod_{j=1}^{m}
			\bigl(1-\eta_j(\tanh(t))^2\bigr)}
		{\bigl(1-u^2(\tanh(t))^2\bigr)^{d+m}}.
	\end{align}
	By~\eqref{eq:flowseries}, 
	we obtain
	\begin{align}\label{eq:flowproduct}
		[t^{2n}]
		\bigl(H(t,u)R_m(\phi(t,u))
		H(-t,u)R_m(\phi(-t,u))\bigr)
		&=
		\sum_{j=0}^{2n}\frac{(-1)^j}{j!(2n-j)!}
		R_{m+j}(u)R_{m+2n-j}(u)
		\notag\\
		&=(-1)^n\mathscr{T}_{m,n}(\mathcal{R};u).
	\end{align}
	Applying~\eqref{eq:G} and~\eqref{eq:pairedkernel} to~\eqref{eq:flowproduct},
	we get
	\begin{align}
		\mathscr{T}_{m,n}(\mathcal{R};u)
		=(-1)^nR_m(u)^2G_{n,d,m}(u^2).
		\label{eq:Tkernel}
	\end{align}
	
	For the fixed $u$, the numbers $\eta_1,\ldots,\eta_m$
	in \eqref{eq:eta} are held fixed.
	Since $d>0$ and $|\eta_j|\leq1$,
	Lemma~\ref{lem:coefficient} 
	shows that all zeros of
	$G_{n,d,m}(z)$ lie in $|z|\leq1$.
	As $|u^2|>1$, we have
$G_{n,d,m}(u^2)\neq0.$
	Since $R_m(u)\neq0$, it follows from \eqref{eq:Tkernel}
	that $\mathscr{T}_{m,n}(\mathcal{R};u)\neq0$.
\end{proof}

\subsection{Boundary cases and proof of Theorem~\ref{thm:unified}~{(i)}}

Recall that $\sigma=\rho=0$ in
Theorem~\ref{thm:unified}~{\rm(i)}.
By \eqref{eq:unified-recurrence}, the sequence
$\mathcal P=(P_j(x))_{j\geq0}$ satisfies $P_0(x)=1$ and
\begin{align*}
	P_{j+1}(x)
	&=\bigl(a_2+(b_1+b_2)(x+\lambda)
	+\tau j(x+\lambda)(a_1+b_1(x+\lambda))\bigr)P_j(x)\\
	&\qquad+(x+\lambda)[1-\tau(x+\lambda)]
	(a_1+b_1(x+\lambda))D_xP_j(x).
\end{align*}
We first consider the boundary cases of this recurrence.

\begin{lemma}
	Let $\mathcal{P}$ be defined by \eqref{eq:unified-recurrence}, under
	the hypotheses of Theorem~\ref{thm:unified}~{(i)}.
	For all integers $m\geq0$ and $n\geq1$,
	the following conclusions hold.
	
	If $a_1=0$ and $b_1>0$, then
	\begin{align}
		\mathscr{T}_{m,n}(\mathcal{P};x)
		=\frac{(-1)^n b_1^{2m+2n}}{n!}
		\left(\frac{b_1+b_2}{b_1}\right)_m^2
		\left(\frac{b_1+b_2}{b_1}+m\right)_n
		(x+\lambda)^{2m+2n}.
		\label{eq:monomialT}
	\end{align}
	If $b_1=\tau=0$, then the degree of $\mathscr{T}_{m,n}(\mathcal{P};x)$
	is $2m+n$, and
	\begin{align}
		[x^{2m+n}]\mathscr{T}_{m,n}(\mathcal{P};x)
		=\frac{(-1)^n a_1^n b_2^{2m+n}}{n!}\neq0.
		\label{eq:boundarynonzero}
	\end{align}
\end{lemma}

\begin{proof}
Suppose that $a_1=0$ and $b_1>0$.
Since $a_1(b_1+b_2)\geq a_2b_1$, we have $a_2=0$.
We claim that, for every $j\geq0$,
\begin{align}\label{eq:monomial}
	P_j(x)
	=b_1^j\left(\frac{b_1+b_2}{b_1}\right)_j
	(x+\lambda)^j.
\end{align}
We prove the claim by induction on $j$.
For $j=0$, we obtain~\eqref{eq:monomial} from $P_0(x)=1$.
Suppose that \eqref{eq:monomial} holds for some $j\geq0$.
Then $(x+\lambda)D_xP_j(x)=jP_j(x)$.
Since $a_1=a_2=\sigma=\rho=0$, substituting this identity
into \eqref{eq:unified-recurrence} gives
\begin{align*}
	P_{j+1}(x)
	&=(b_1+b_2+b_1j)(x+\lambda)P_j(x)+\tau b_1j(x+\lambda)^2P_j(x)
	-\tau b_1j(x+\lambda)^2P_j(x)\\
	&=(b_1(j+1)+b_2)(x+\lambda)P_j(x)\\
	&=b_1^{j+1}\left(\frac{b_1+b_2}{b_1}\right)_{j+1}
	(x+\lambda)^{j+1}.
\end{align*}
Thus \eqref{eq:monomial} holds for $j+1$,
completing the induction.
	Substituting \eqref{eq:monomial} into \eqref{eq:turan}
	and using $(\alpha)_{m+j}=(\alpha)_m(\alpha+m)_j$,
	we obtain
	\begin{align*}
		\mathscr{T}_{m,n}(\mathcal{P};x)
		&=(-1)^n b_1^{2m+2n}
		\left(\frac{b_1+b_2}{b_1}\right)_m^2
		(x+\lambda)^{2m+2n}\sum_{j=0}^{2n}
		\frac{(-1)^j
			\left(\frac{b_1+b_2}{b_1}+m\right)_j
			\left(\frac{b_1+b_2}{b_1}+m\right)_{2n-j}}
		{j!(2n-j)!}\\
		&=(-1)^n b_1^{2m+2n}
		\left(\frac{b_1+b_2}{b_1}\right)_m^2
		(x+\lambda)^{2m+2n}[t^{2n}]
		(1-t)^{-\frac{b_1+b_2}{b_1}-m}
		(1+t)^{-\frac{b_1+b_2}{b_1}-m}
		\\
		&=(-1)^n b_1^{2m+2n}
		\left(\frac{b_1+b_2}{b_1}\right)_m^2
		(x+\lambda)^{2m+2n}[t^{2n}](1-t^2)^{-\frac{b_1+b_2}{b_1}-m}
		\\
		&=(-1)^n b_1^{2m+2n}
		\left(\frac{b_1+b_2}{b_1}\right)_m^2
		(x+\lambda)^{2m+2n}\frac{1}{n!}
		\left(\frac{b_1+b_2}{b_1}+m\right)_n.
	\end{align*}
	This proves~\eqref{eq:monomialT}.
	
	Next suppose that $b_1=\tau=0$.
	From the inequalities $a_1+b_1>0$ and
	$b_1+b_2+\tau a_2>0$ in \eqref{eq:unified-parameters},
	it follows that $a_1,b_2>0$.
	Define
	\begin{align}\label{eq:boundaryF}
		F(t,x)=\exp\left(a_2t+
		\frac{b_2}{a_1}(x+\lambda)(e^{a_1t}-1)\right).
	\end{align}
	Since
	$\partial_tF(t,x)=(a_2+b_2(x+\lambda)e^{a_1t})F(t,x)$,
	we obtain
	\begin{align}\label{eq:boundaryshift}
		\partial_t\bigl(F(t,x)P_j((x+\lambda)e^{a_1t}-\lambda)\bigr)
		&=\bigl(\partial_tF(t,x)\bigr)
		P_j((x+\lambda)e^{a_1t}-\lambda)\notag\\
		&\qquad+a_1(x+\lambda)e^{a_1t}F(t,x)
		\left.D_vP_j(v)\right|_{v=(x+\lambda)e^{a_1t}-\lambda}
		\notag\\
		&=F(t,x)\biggl(
		(a_2+b_2(x+\lambda)e^{a_1t})
		P_j((x+\lambda)e^{a_1t}-\lambda)\notag\\
		&\qquad+a_1(x+\lambda)e^{a_1t}
		\left.D_vP_j(v)\right|_{v=(x+\lambda)e^{a_1t}-\lambda}
		\biggr)\notag\\
		&=F(t,x)P_{j+1}((x+\lambda)e^{a_1t}-\lambda),
	\end{align}
	where the last equality follows from
	\eqref{eq:unified-recurrence} with $b_1=\tau=\sigma=\rho=0$.
	Since $P_0(x)=1$, repeated application of
	\eqref{eq:boundaryshift} gives
	\begin{align}\label{eq:boundaryderivative}
		\partial_t^mF(t,x)
		=F(t,x)P_m((x+\lambda)e^{a_1t}-\lambda).
	\end{align}
Evaluating \eqref{eq:boundaryderivative} at $t=0$ gives
$
\left.\partial_t^mF(t,x)\right|_{t=0}
=F(0,x)P_m(x)=P_m(x).
$
Taylor's formula then yields
\begin{align}\label{eq:boundaryseries}
	F(t,x)=\sum_{j=0}^{\infty}P_j(x)\frac{t^j}{j!}.
\end{align}
Differentiating both sides of~\eqref{eq:boundaryseries}
$m$ times with respect to $t$ and using~\eqref{eq:boundaryderivative},
we obtain
\begin{align}\label{eq:boundaryshiftedseries}
	F(t,x)P_m((x+\lambda)e^{a_1t}-\lambda)
	&=\partial_t^mF(t,x)
	=\sum_{j=0}^{\infty}P_{m+j}(x)\frac{t^j}{j!}.
\end{align}
It also follows from \eqref{eq:boundaryF} that
\begin{align}\label{eq:boundaryFpair}
	F(t,x)F(-t,x)
	&=\exp\left(\frac{b_2}{a_1}(x+\lambda)
	(e^{a_1t}+e^{-a_1t}-2)\right)\notag\\
	&=\exp\left(\frac{2b_2}{a_1}(x+\lambda)
	(\cosh(a_1t)-1)\right).
\end{align}
By \eqref{eq:turan}, we have
\begin{align}\label{eq:boundarypair}
	(-1)^n\mathscr{T}_{m,n}(\mathcal{P};x)
	&=\sum_{j=0}^{2n}\frac{(-1)^j}{j!(2n-j)!}
	P_{m+j}(x)P_{m+2n-j}(x)\notag\\
	&=[t^{2n}]
	\left(\sum_{r=0}^{\infty}P_{m+r}(x)\frac{t^r}{r!}\right)
	\left(\sum_{s=0}^{\infty}P_{m+s}(x)\frac{(-t)^s}{s!}\right)
	\notag\\
	&=[t^{2n}]\biggl(
	\exp\left(\frac{2b_2}{a_1}(x+\lambda)
	(\cosh(a_1t)-1)\right)\notag\\
	&\qquad\times P_m((x+\lambda)e^{a_1t}-\lambda)
	P_m((x+\lambda)e^{-a_1t}-\lambda)\biggr),
\end{align}
where  the last equality follows from
\eqref{eq:boundaryshiftedseries} and \eqref{eq:boundaryFpair}.
We now determine the degree and leading coefficient of
$\mathscr{T}_{m,n}(\mathcal{P};x)$.
Since $\sigma=\rho=b_1=\tau=0$,
the recurrence in \eqref{eq:unified-recurrence} reduces to
\begin{align*}
	P_{j+1}(x)
	=[a_2+b_2(x+\lambda)]P_j(x)
	+a_1(x+\lambda)D_xP_j(x).
\end{align*}
Only the term $b_2xP_j(x)$ increases the degree.
Starting from $P_0(x)=1$, induction therefore gives
$\deg P_m(x)=m$ and $[x^m]P_m(x)=b_2^m$.
Since
$(x+\lambda)e^{\pm a_1t}-\lambda
=e^{\pm a_1t}x+\lambda(e^{\pm a_1t}-1)$,
we obtain
\begin{align*}
	[x^m]P_m((x+\lambda)e^{a_1t}-\lambda)
	&=b_2^m(e^{a_1t})^m=b_2^me^{ma_1t},\\
	[x^m]P_m((x+\lambda)e^{-a_1t}-\lambda)
	&=b_2^m(e^{-a_1t})^m=b_2^me^{-ma_1t}.
\end{align*}
Thus
\begin{align}\label{eq:boundaryproductleading}
	[x^{2m}]
	\bigl[P_m((x+\lambda)e^{a_1t}-\lambda)
	P_m((x+\lambda)e^{-a_1t}-\lambda)\bigr]=(b_2^me^{ma_1t})(b_2^me^{-ma_1t})
	=b_2^{2m}.
\end{align}

Expanding the exponential factor in \eqref{eq:boundarypair} gives
\begin{align}\label{eq:boundaryexpansion}
	\exp\left(\frac{2b_2}{a_1}(x+\lambda)
	(\cosh(a_1t)-1)\right)
	=\sum_{r=0}^{\infty}\frac{(x+\lambda)^r}{r!}
	\left(\frac{2b_2}{a_1}(\cosh(a_1t)-1)\right)^r.
\end{align}
Since $\cosh(a_1t)-1=a_1^2t^2/2+O(t^4)$, we have
\begin{align*}
	\frac{2b_2}{a_1}(\cosh(a_1t)-1)
	=a_1b_2t^2+O(t^4).
\end{align*}
Thus the $r$th term in \eqref{eq:boundaryexpansion}
starts at degree $2r$ in $t$.
The two polynomial factors in \eqref{eq:boundarypair}
are analytic at $t=0$, so only $r\leq n$ can contribute
to the coefficient of $t^{2n}$.
For each such $r$, the degree in $x$ is at most
$2m+r\leq2m+n$.

To obtain the coefficient of $x^{2m+n}$, we take $r=n$
in \eqref{eq:boundaryexpansion} and the coefficient of
$x^{2m}$ in \eqref{eq:boundaryproductleading}.
Consequently, we obtain from \eqref{eq:boundarypair} that
\begin{align*}
	[x^{2m+n}]\mathscr{T}_{m,n}(\mathcal{P};x)
	&=\frac{(-1)^nb_2^{2m}}{n!}
	[t^{2n}]
	\left(\frac{2b_2}{a_1}(\cosh(a_1t)-1)\right)^n\\
	&=\frac{(-1)^nb_2^{2m}}{n!}(a_1b_2)^n\\
	&=\frac{(-1)^na_1^nb_2^{2m+n}}{n!}\\
	&\neq0.
\end{align*}
Therefore $\deg\mathscr{T}_{m,n}(\mathcal{P};x)=2m+n$,
and \eqref{eq:boundarynonzero} follows.
\end{proof}

\begin{proof}[\textbf{Proof of Theorem~\ref{thm:unified} {(i)}}]
	We consider two cases.
	Case~1: $a_1>0$ and $b_1+\tau a_1>0$;
	Case~2: $a_1(b_1+\tau a_1)=0$.
	The proof proceeds in four steps.
	
	For Case~1, we first express $P_j(x)$ in terms of $R_j(u)$
	in Step~1.
	Using this representation and Proposition~\ref{lem:RTuran},
	we prove in Step~2 that
	$\mathscr T_{m,n}(\mathcal P;x)\neq0$
	for $\Rea x>-\lambda$ when $1-\tau(x+\lambda)\neq0$.
	In Step~3, we verify
	$\mathscr T_{m,n}(\mathcal P;x)\neq0$
	for $\Rea x>-\lambda$ when $1-\tau(x+\lambda)=0$
	by direct evaluation.
	
	For Case~2, we prove
	$\mathscr T_{m,n}(\mathcal P;x)\neq0$
	for $\Rea x>-\lambda$ in Step~4
	using an explicit formula and Hurwitz's theorem.
	
\noindent\textbf{Step 1: Expressing $P_j(x)$
		in terms of $R_j(u)$.}
Suppose that $a_1>0$ and $b_1+\tau a_1>0$.
Recall that $\sigma=\rho=0$ in
Theorem~\ref{thm:unified}~{\rm(i)}.
Thus, by~\eqref{eq:unified-recurrence}, we have
\begin{align*}
	P_{j+1}(x)
	&=\bigl(a_2+(b_1+b_2)(x+\lambda)
	+\tau j(x+\lambda)(a_1+b_1(x+\lambda))\bigr)P_j(x)\\
	&\quad+(x+\lambda)(1-\tau(x+\lambda))
	(a_1+b_1(x+\lambda))D_xP_j(x).
\end{align*}
Setting
\begin{align}\label{eq:dh-d}
	d=\frac{b_1+b_2+\tau a_2}{b_1+\tau a_1},
	\qquad
	h=\frac{2a_2}{a_1}-d
\end{align}
in \eqref{eq:Rrecurrence}, we obtain the following recurrence
for $\mathcal R=(R_j(u))_{j\geq0}$:
\begin{align}\label{eq:direct-normalized-recurrence}
	R_{j+1}(u)
	&=(u^2-1)D_uR_j(u)
	+\left(
	\frac{b_1+b_2+\tau a_2}{b_1+\tau a_1}(u-1)
	+\frac{2a_2}{a_1}
	\right)R_j(u),
\end{align}
where $R_0(u)=1$.

We claim that, whenever $1-\tau(x+\lambda)\neq0$,
\begin{align}\label{eq:unified-normalization}
	P_j(x)
	=\left(\frac{a_1(1-\tau(x+\lambda))}{2}\right)^j
	R_j\!\left(
	\frac{1+(2b_1/a_1+\tau)(x+\lambda)}
	{1-\tau(x+\lambda)}\right).
\end{align}
We prove \eqref{eq:unified-normalization} by induction on $j$.
The case $j=0$ is immediate because $P_0(x)=R_0(x)=1$. 
Suppose that \eqref{eq:unified-normalization} holds
for some $j\geq0$.
Differentiating with respect to $x$, we obtain
\begin{align}\label{eq:DxPj}
	D_xP_j(x)
	&=\frac{2(b_1+\tau a_1)}
	{a_1(1-\tau(x+\lambda))^2}
	\left(\frac{a_1(1-\tau(x+\lambda))}{2}\right)^j
	\left.D_uR_j(u)\right|_{
		u=\frac{1+(2b_1/a_1+\tau)(x+\lambda)}
		{1-\tau(x+\lambda)}}\notag\\
		&\qquad-\frac{j\tau}{1-\tau(x+\lambda)}P_j(x).
\end{align}
Substituting \eqref{eq:DxPj} into
\eqref{eq:unified-recurrence}, we find that the resulting term
$-\tau j(x+\lambda)(a_1+b_1(x+\lambda))P_j(x)$
cancels the term
$\tau j(x+\lambda)(a_1+b_1(x+\lambda))P_j(x)$
on the right-hand side of \eqref{eq:unified-recurrence}.
Consequently,
\begin{align*}
	P_{j+1}(x)
	&=(a_2+(b_1+b_2)(x+\lambda))P_j(x)
	+\frac{2(b_1+\tau a_1)(x+\lambda)
		(a_1+b_1(x+\lambda))}
	{a_1(1-\tau(x+\lambda))}\\
	&\qquad\qquad\times\left(\frac{a_1(1-\tau(x+\lambda))}{2}\right)^j
	\left.D_uR_j(u)\right|_{
		u=\frac{1+(2b_1/a_1+\tau)(x+\lambda)}
		{1-\tau(x+\lambda)}}.
\end{align*}

To apply \eqref{eq:direct-normalized-recurrence},
we calculate
\begin{align*}
	\left(
	\frac{1+(2b_1/a_1+\tau)(x+\lambda)}
	{1-\tau(x+\lambda)}\right)^2-1=
	\frac{4(b_1+\tau a_1)(x+\lambda)
		(a_1+b_1(x+\lambda))}
	{a_1^2(1-\tau(x+\lambda))^2},
\end{align*}
and
\begin{align*}
	&\frac{b_1+b_2+\tau a_2}{b_1+\tau a_1}
	\left(
	\frac{1+(2b_1/a_1+\tau)(x+\lambda)}
	{1-\tau(x+\lambda)}-1\right)
	+\frac{2a_2}{a_1}\\
	=&
	\frac{2(b_1+b_2+\tau a_2)(x+\lambda)}
	{a_1(1-\tau(x+\lambda))}
	+\frac{2a_2}{a_1}\\
	=&
	\frac{2(a_2+(b_1+b_2)(x+\lambda))}
	{a_1(1-\tau(x+\lambda))}.
\end{align*}
We now substitute \eqref{eq:unified-normalization}
for $P_j(x)$ in the expression for $P_{j+1}(x)$
and factor out
$\bigl(a_1(1-\tau(x+\lambda))/2\bigr)^{j+1}$.
Using \eqref{eq:direct-normalized-recurrence}, we obtain 
\begin{align*}
	P_{j+1}(x)
	&=\left(\frac{a_1(1-\tau(x+\lambda))}{2}\right)^{j+1}
	R_{j+1}\!\left(
	\frac{1+(2b_1/a_1+\tau)(x+\lambda)}
	{1-\tau(x+\lambda)}\right).
\end{align*}
Hence \eqref{eq:unified-normalization} holds for $j+1$,
and induction proves \eqref{eq:unified-normalization}
for all $j\geq0$.
	
	\noindent\textbf{Step 2: $\mathscr{T}_{m,n}(\mathcal{P};x)\neq0$
		when $\Rea x>-\lambda$ and $1-\tau(x+\lambda)\neq0$.}
	For $0\leq j\leq2n$, applying
	\eqref{eq:unified-normalization} to
	$P_{m+j}(x)P_{m+2n-j}(x)$ gives the factor
	\begin{align*}
		\left(\frac{a_1(1-\tau(x+\lambda))}{2}\right)^{m+j}
		\left(\frac{a_1(1-\tau(x+\lambda))}{2}\right)^{m+2n-j}
		=
		\left(\frac{a_1(1-\tau(x+\lambda))}{2}\right)^{2m+2n}.
	\end{align*}
	This factor is independent of $j$, so it can be taken outside
	the sum in \eqref{eq:turan}. Hence
	\begin{align}\label{eq:unified-turan-transform}
		\mathscr{T}_{m,n}(\mathcal{P};x)
		=\left(\frac{a_1(1-\tau(x+\lambda))}{2}\right)^{2m+2n}
		\mathscr{T}_{m,n}\!\left(\mathcal{R};
		\frac{1+(2b_1/a_1+\tau)(x+\lambda)}
		{1-\tau(x+\lambda)}\right).
	\end{align}
	Moreover, direct expansion gives
	\begin{align*}
		&\left|1+\left(\frac{2b_1}{a_1}+\tau\right)(x+\lambda)\right|^2
		-|1-\tau(x+\lambda)|^2=\frac{4(b_1+\tau a_1)}{a_1^2}
		\bigl(a_1\Rea(x+\lambda)+b_1|x+\lambda|^2\bigr)>0,
	\end{align*}
	where strict positivity follows from
	$\Rea(x+\lambda)>0$, $a_1>0$, $b_1\geq0$, and
	$b_1+\tau a_1>0$. Since $1-\tau(x+\lambda)\neq0$, 
	we have 
	\begin{align*}
		\left|
		\frac{1+(2b_1/a_1+\tau)(x+\lambda)}
		{1-\tau(x+\lambda)}
		\right|>1.
	\end{align*}
	By \eqref{eq:unified-parameters} and \eqref{eq:dh-d}, we have $d>0$ and
	\begin{align*}
		d-\frac{a_2}{a_1}
		=\frac{a_1(b_1+b_2)-a_2b_1}
		{a_1(b_1+\tau a_1)}
		\geq0.
	\end{align*}
	Since $a_2/a_1\geq0$, we obtain $-d\leq h\leq d$.
	Since $a_1>0$ and $1-\tau(x+\lambda)\neq0$,
	Proposition~\ref{lem:RTuran} and
	\eqref{eq:unified-turan-transform} imply that
	$\mathscr T_{m,n}(\mathcal P;x)\neq0$.
	
	Moreover, if $\tau=0$, then $1-\tau(x+\lambda)=1\neq0$.
	Thus Step 2 gives $\mathscr T_{m,n}(\mathcal P;x)\neq0$
	for $\Rea x>-\lambda$ when $a_1>0$, $b_1>0$, and $\tau=0$.
	
\noindent\textbf{Step 3:
	$\mathscr{T}_{m,n}(\mathcal{P};x)\neq0$
	when $\Rea x>-\lambda$ and $1-\tau(x+\lambda)=0$.} 
	Since $\tau\geq0$ and $1-\tau(x+\lambda)=0$,
	we have $\tau>0$. 
	By \eqref{eq:unified-parameters}, we get  $b_1+b_2+\tau a_2>0$.
	Since $b_1+\tau a_1\geq\tau a_1>0$, we have
	\begin{align}\label{eq:direct-admissibility}
		\frac{b_1+b_2+\tau a_2}{b_1+\tau a_1}>0.
	\end{align}
	At $x=1/\tau-\lambda$, the derivative term in
	\eqref{eq:unified-recurrence} vanishes, so
	\begin{align}\label{eq:direct-pole-recurrence}
		P_{j+1}(1/\tau-\lambda)
		&=\left(a_1+\frac{b_1}{\tau}\right)
		\left(\frac{b_1+b_2+\tau a_2}{b_1+\tau a_1}+j\right)
		P_j(1/\tau-\lambda).
	\end{align}
	We prove by induction that
	\begin{align}\label{eq:unified-pole-values}
		P_j(1/\tau-\lambda)
		=\left(a_1+\frac{b_1}{\tau}\right)^j
		\left(\frac{b_1+b_2+\tau a_2}{b_1+\tau a_1}\right)_j.
	\end{align}
	For $j=0$, both sides equal $1$.
	Suppose that \eqref{eq:unified-pole-values} holds for some
	$j\geq0$. Then, using \eqref{eq:direct-pole-recurrence}, we obtain 
	\begin{align*}
		P_{j+1}(1/\tau-\lambda)
		&=\left(a_1+\frac{b_1}{\tau}\right)^{j+1}
		\left(\frac{b_1+b_2+\tau a_2}{b_1+\tau a_1}+j\right)
		\left(\frac{b_1+b_2+\tau a_2}{b_1+\tau a_1}\right)_j\\
		&=\left(a_1+\frac{b_1}{\tau}\right)^{j+1}
		\left(\frac{b_1+b_2+\tau a_2}{b_1+\tau a_1}\right)_{j+1}.
	\end{align*}
	This proves~\eqref{eq:unified-pole-values}.
	
	Substituting \eqref{eq:unified-pole-values} into
	\eqref{eq:turan} and applying the binomial theorem, we obtain
	\begin{align*}
		\mathscr{T}_{m,n}(\mathcal{P};1/\tau-\lambda)
		&=\frac{(-1)^n}{n!}
		\left(a_1+\frac{b_1}{\tau}\right)^{2m+2n}
		\left(\frac{b_1+b_2+\tau a_2}{b_1+\tau a_1}\right)_m^2
		\notag\\
		&\quad\times
		\left(\frac{b_1+b_2+\tau a_2}{b_1+\tau a_1}+m\right)_n
		\neq0.
	\end{align*}
	Indeed, all factors except $(-1)^n$ are positive by
	\eqref{eq:direct-admissibility} and $a_1>0$.
	Combining this with Step 2, we obtain
	$\mathscr{T}_{m,n}(\mathcal{P};x)\neq0$ for $\Rea x>-\lambda$
	whenever $a_1>0$ and $b_1+\tau a_1>0$.
	
	\noindent\textbf{Step 4: $a_1(b_1+\tau a_1)=0$.} 
If $a_1=0$, the inequalities $a_1+b_1>0$ and
$a_1(b_1+b_2)\geq a_2b_1$ in \eqref{eq:unified-parameters}
reduce to $b_1>0$ and $a_2b_1\leq0$, respectively.
Since $b_1>0$ and $a_2\geq0$, it follows that $a_2=0$.
The polynomial $\mathscr{T}_{m,n}(\mathcal{P};x)$ is a nonzero
multiple of $(x+\lambda)^{2m+2n}$ by \eqref{eq:monomialT}.
Hence all its zeros are at $-\lambda$.
	
	Now suppose that $a_1>0$ and $b_1+\tau a_1=0$.
	Since $b_1,\tau\geq0$, we have $b_1=\tau=0$.
	Using \eqref{eq:unified-parameters}, we obtain $b_2>0$.
For $\varepsilon>0$, let $\mathcal{P}^{(\varepsilon)}$ be defined
by \eqref{eq:unified-recurrence} with $b_1$ replaced by $\varepsilon$.
Since $a_1b_2>0$, we have
\begin{align*}
	a_1(b_2+\varepsilon)-a_2\varepsilon
	=a_1b_2+(a_1-a_2)\varepsilon>0
\end{align*}
for all sufficiently small $\varepsilon>0$.
Together with $b_2+\varepsilon>0$ and $a_1+\varepsilon>0$,
this verifies \eqref{eq:unified-parameters} after the replacement.

With $b_1=\varepsilon$ and $\tau=0$, the prefactor in
\eqref{eq:unified-turan-transform} is $(a_1/2)^{2m+2n}>0$.
For $\Rea x>-\lambda$, the transformed argument satisfies
\begin{align*}
	\left|1+\frac{2\varepsilon}{a_1}(x+\lambda)\right|
	\geq 1+\frac{2\varepsilon}{a_1}\Rea(x+\lambda)>1.
\end{align*}
Thus \eqref{eq:unified-turan-transform} and
Proposition~\ref{lem:RTuran} imply that
$\mathscr{T}_{m,n}(\mathcal{P}^{(\varepsilon)};x)\neq0$
for $\Rea x>-\lambda$.
	
	By \eqref{eq:unified-recurrence} and  $\tau=0$,
	we have $\deg P_j^{(\varepsilon)}\leq j$ and hence
	$\deg\mathscr{T}_{m,n}(\mathcal{P}^{(\varepsilon)};x)\leq2m+2n$.
	The recurrence also gives coefficientwise convergence to
	$\mathscr{T}_{m,n}(\mathcal{P};x)$ as $\varepsilon\to0$.
	The degree bound therefore implies uniform convergence
	on compact subsets of $\mathbb C$.
	The coefficient in \eqref{eq:boundarynonzero} is nonzero,
	so the limit is not identically zero.
	Hurwitz theorem now yields
	$\mathscr{T}_{m,n}(\mathcal{P};x)\neq0$ for $\Rea x>-\lambda$.
\end{proof}

\section{Proof of Theorem~\ref{thm:unified}~(ii)}\label{sec:proof-ii}

Recall that, for $\alpha>-1$, the generalized Laguerre
polynomials are given by 
\begin{align}\label{eq:generalized-laguerre}
	L_j^{(\alpha)}(x)
	=\sum_{r=0}^j(-1)^r\binom{j+\alpha}{j-r}\frac{x^r}{r!},
	\qquad j\geq0,
\end{align}
where $\binom{j+\alpha}{j-r}=(\alpha+r+1)_{j-r}/(j-r)!$.
For $\alpha>-1$, $L_j^{(\alpha)}$ has $j$ simple positive zeros
and satisfies the differential recurrence
\begin{align}\label{eq:generalized-laguerre-recurrence}
	(j+1)L_{j+1}^{(\alpha)}(-x)
	=(j+\alpha+1+x)L_j^{(\alpha)}(-x)
	+xD_xL_j^{(\alpha)}(-x);
\end{align}
see \cite[Theorem~3.3.1 and Section~5.1]{Szego}.
The ordinary Laguerre polynomials are
$L_j(x)=L_j^{(0)}(x)$.

Under the hypotheses of Theorem~\ref{thm:unified}~{\rm(ii)},
we have $b_2>0$ and $a_2/a_1-1>-1$.
By~\eqref{eq:unified-recurrence}, we obtain 
\begin{align}\label{eq:laguerre-branch-recurrence}
	(1+\sigma j)P_{j+1}(x)
	=(a_2+a_1j+b_2(x+\lambda))P_j(x)
	+a_1(x+\lambda)D_xP_j(x).
\end{align}
We claim that
\begin{align}\label{eq:laguerre-branch-explicit}
	P_j(x)
	=\left(\frac{a_1}{\sigma}\right)^j
	\frac{j!}{(1/\sigma)_j}
	L_j^{(a_2/a_1-1)}
	\left(-\frac{b_2}{a_1}(x+\lambda)\right).
\end{align}
We prove this claim by induction on $j$.
The case $j=0$ follows from $P_0(x)=1$ and $L_0^{(\alpha)}(x)=1$.
Assume that \eqref{eq:laguerre-branch-explicit} holds for $j$.
Using \eqref{eq:generalized-laguerre-recurrence}
and the chain rule, we obtain
\begin{align}\label{eq:laguerre-scaled-recurrence}
	(j+1)L_{j+1}^{(a_2/a_1-1)}
	\left(-\frac{b_2}{a_1}(x+\lambda)\right)&=\left(j+\frac{a_2}{a_1}
	+\frac{b_2}{a_1}(x+\lambda)\right)
	L_j^{(a_2/a_1-1)}
	\left(-\frac{b_2}{a_1}(x+\lambda)\right)\notag\\
	&\qquad+(x+\lambda)D_xL_j^{(a_2/a_1-1)}
	\left(-\frac{b_2}{a_1}(x+\lambda)\right).
\end{align}
Substituting \eqref{eq:laguerre-branch-explicit} into
\eqref{eq:laguerre-branch-recurrence} and using
\eqref{eq:laguerre-scaled-recurrence}, we obtain
\begin{align*}
	P_{j+1}(x)
	&=\frac{(a_2+a_1j+b_2(x+\lambda))P_j(x)
		+a_1(x+\lambda)D_xP_j(x)}{1+\sigma j}\\
	&=\frac{1}{1+\sigma j}
	\left(\frac{a_1}{\sigma}\right)^j
	\frac{j!}{(1/\sigma)_j}
	\biggl(
	(a_2+a_1j+b_2(x+\lambda))
	L_j^{(a_2/a_1-1)}\left(-\frac{b_2}{a_1}(x+\lambda)\right)\\
	&\qquad\qquad
	+a_1(x+\lambda)D_xL_j^{(a_2/a_1-1)}
	\left(-\frac{b_2}{a_1}(x+\lambda)\right)
	\biggr)\\
	&=\frac{a_1(j+1)}{1+\sigma j}
	\left(\frac{a_1}{\sigma}\right)^j
	\frac{j!}{(1/\sigma)_j}
	L_{j+1}^{(a_2/a_1-1)}
	\left(-\frac{b_2}{a_1}(x+\lambda)\right)\\
	&=\left(\frac{a_1}{\sigma}\right)^{j+1}
	\frac{(j+1)!}{(1/\sigma)_{j+1}}
	L_{j+1}^{(a_2/a_1-1)}
	\left(-\frac{b_2}{a_1}(x+\lambda)\right),
\end{align*}
where the last equality uses
$(1/\sigma)_{j+1}=(1/\sigma+j)(1/\sigma)_j$.
Thus \eqref{eq:laguerre-branch-explicit} holds for $j+1$,
completing the induction.

\begin{lemma}\label{lem:nonzero}
	Under the hypotheses of Theorem~\ref{thm:unified}~{\rm(ii)},
	let $\mathcal P=(P_j(x))_{j\geq0}$ be the polynomial sequence
	given by \eqref{eq:laguerre-branch-explicit}.
	Then, for all integers $m\geq0$ and $n\geq1$, we have
	\begin{align}\label{eq:nonzerocoefficient}
		[x^{2m+2n}]\mathscr{T}_{m,n}(\mathcal P;x)
		=\frac{(b_2/\sigma)^{2m+2n}}
		{n!(1/\sigma)_m^2(1/\sigma+m)_n(1/\sigma+m)_{2n}}>0.
	\end{align}
\end{lemma}

\begin{proof}
	For $b>0$, using
	$\binom{b+2n-1}{r}=(b)_{2n}/(r!(b)_{2n-r})$, we obtain
	\begin{align}\label{eq:pochhammer-alternating}
		\sum_{r=0}^{2n}
		\frac{(-1)^r}{r!(2n-r)!(b)_r(b)_{2n-r}}
		&=\frac{[z^{2n}](1-z)^{b+2n-1}(1+z)^{b+2n-1}}
		{(b)_{2n}^2}\notag\\
		&=\frac{[z^{2n}](1-z^2)^{b+2n-1}}{(b)_{2n}^2}
		\notag\\
		&=\frac{(-1)^n\binom{b+2n-1}{n}}{(b)_{2n}^2}\notag\\
		&=\frac{(-1)^n}{n!(b)_n(b)_{2n}}.
	\end{align}
	
	Next, by~\eqref{eq:laguerre-branch-explicit}, we have 
	\begin{align*}
		[x^j]P_j(x)
		=\left(\frac{a_1}{\sigma}\right)^j
		\frac{j!}{(1/\sigma)_j}
		\frac{(b_2/a_1)^j}{j!}
		=\frac{(b_2/\sigma)^j}{(1/\sigma)_j}.
	\end{align*}
	Substituting these leading coefficients into \eqref{eq:turan}
	and using
	$(1/\sigma)_{m+r}=(1/\sigma)_m(1/\sigma+m)_r$,
	we get 
	\begin{align*}
		[x^{2m+2n}]\mathscr{T}_{m,n}(\mathcal{P};x)
		&=
		\frac{(-1)^n(b_2/\sigma)^{2m+2n}}{(1/\sigma)_m^2}
		\sum_{r=0}^{2n}
		\frac{(-1)^r}
		{r!(2n-r)!(1/\sigma+m)_r(1/\sigma+m)_{2n-r}}\\
		&\quad=
		\frac{(b_2/\sigma)^{2m+2n}}
		{n!(1/\sigma)_m^2(1/\sigma+m)_n(1/\sigma+m)_{2n}}\\
		&\quad=
		\frac{(b_2/\sigma)^{2m+2n}}
		{n!(1/\sigma)_{m+n}(1/\sigma)_{m+2n}}>0,
	\end{align*}
	where the second equality follows from
	\eqref{eq:pochhammer-alternating}.
	This proves \eqref{eq:nonzerocoefficient}.
\end{proof}

The next two lemmas control the zeros of the approximating polynomials.

\begin{lemma}\label{lem:sector}
	Let $p(t)$ be a polynomial of degree $h$ with $p(0)\neq0$.
	Suppose that all its zeros lie in the sector
	\begin{align*}
		\left\{
		re^{\mathrm i\theta}:
		r>0,\ \beta\leq\theta\leq\beta+\alpha
		\right\},
	\end{align*}
	where $\beta\in\mathbb R$ and $0\leq\alpha<\pi/2$.
	Then, for every integer $n$ with $0\leq n\leq h$, we have
	\begin{align}\label{eq:sector}
		(-1)^n[t^{2n}]p(t)p(-t)\neq0.
	\end{align}
\end{lemma}

\begin{proof}
	The case $n=0$ follows from $p(0)\ne0$, so assume $1\leq n\leq h$.
	Let $r_1,\ldots,r_h$ be the zeros of $p(t)$, counted with multiplicity.
	Since $p(0)\ne0$, we have 
\begin{align*}
	p(t)p(-t)
	=p(0)^2\prod_{j=1}^h
	\left(1-\frac{t}{r_j}\right)
	\left(1+\frac{t}{r_j}\right)
	=p(0)^2\prod_{j=1}^h
	\left(1-\frac{t^2}{r_j^2}\right).
\end{align*}
	Consequently,
	\begin{align}\label{eq:reciprocal}
		(-1)^n[t^{2n}]p(t)p(-t)
		=p(0)^2\sum_{1\leq j_1<\cdots<j_n\leq h}
		\frac1{r_{j_1}^2\cdots r_{j_n}^2}.
	\end{align}
	Write $r_j=|r_j|e^{\mathrm i\theta_j}$ with
	$\beta\leq\theta_j\leq\beta+\alpha$, and define
	\begin{align}\label{eq:rotatedreciprocals}
		w_j=e^{\mathrm i(2\beta+\alpha)}r_j^{-2}
		=|r_j|^{-2}e^{\mathrm i(2\beta+\alpha-2\theta_j)}.
	\end{align}
	The bounds on $\theta_j$ imply 
		$-\alpha\leq 2\beta+\alpha-2\theta_j\leq\alpha.$ 
	Since $\alpha<\pi/2$, we have $\Rea w_j>0$ for every $j$.
	Thus all zeros of $B(z)=\prod_{j=1}^h(z+w_j)$ lie in the open
	left half-plane. The polynomial $D_z^{h-n}B(z)$ has degree
	$n\geq1$. By the Gauss--Lucas theorem, all zeros of $D_z^{h-n}B(z)$
	lie in the open left half-plane. 
	Consequently,
	\begin{align*}
		\sum_{1\leq j_1<\cdots<j_n\leq h}w_{j_1}\cdots w_{j_n}
		=[z^{h-n}]B(z)
		=\frac1{(h-n)!}\left.D_z^{h-n}B(z)\right|_{z=0}\ne0.
	\end{align*}
Combining \eqref{eq:reciprocal} and
\eqref{eq:rotatedreciprocals}, we obtain
\begin{align*}
	(-1)^n[t^{2n}]p(t)p(-t)
	&=p(0)^2e^{-\mathrm i n(2\beta+\alpha)}
	\sum_{1\leq j_1<\cdots<j_n\leq h}
	w_{j_1}\cdots w_{j_n}
	\neq0.
\end{align*}
This proves \eqref{eq:sector}.
\end{proof}

The next lemma is a direct consequence of
\cite[Theorem~(6,1)$'$, p.~22]{Marden}.
We include a short proof for completeness.
\begin{lemma}\label{lem:euler-sector}
	Let $p(t)$ be a polynomial with $p(0)\ne0$, all of whose zeros lie
	in a sector of opening less than $\pi/2$ contained in the open
	left half-plane. For every $s>0$, the polynomial
	$sp(t)+tD_tp(t)$ has the same degree as $p$ and all its zeros
	in the same sector.
\end{lemma}
\begin{proof}
	The polynomial $sp(t)+tD_tp(t)$ has the same degree as $p$,
	and $sp(0)\neq0$.
	Let $r_1,\ldots,r_h$ be the zeros of $p(t)$, counted with
	multiplicity, and let $z$ be a zero of $sp(t)+tD_tp(t)$.
	If $p(z)=0$, then $z$ is one of the $r_i$.
	Otherwise, $z\neq0$, and division by $zp(z)$ gives
	\begin{align*}
		0=\frac{s}{z}+\frac{\left.D_t p(t)\right|_{t=z}}{p(z)}
		=\frac{s}{z}+\sum_{i=1}^h\frac1{z-r_i}.
	\end{align*}
	Lemma~\ref{lem:convex} implies that $z$ belongs to the
	convex hull of $0,r_1,\ldots,r_h$.
	The sector containing the zeros of $p$, together with
	its vertex $0$, is convex.
	Hence $z$ lies in that sector or equals $0$.
	Since $z\neq0$, the conclusion follows.
\end{proof}

\begin{proof}[\textbf{Proof of Theorem~\ref{thm:unified}~{(ii)}}]
	Fix $m\geq0$ and $n\geq1$, and let
	$N=a_2/a_1-1/\sigma\in\mathbb Z_{\geq0}$.
	The proof proceeds in two steps.
	In Step~1, we construct polynomials $H_k(x)$ satisfying
	$H_k(x)\neq0$ for $\Rea x>-\lambda$.
	In Step~2, we prove that
	$\lim_{k\to\infty}H_k(x)
		=\mathscr T_{m,n}(\mathcal P;x),$
	uniformly on compact subsets of $\mathbb C$.
	We then apply Hurwitz's theorem to obtain
	$\mathscr T_{m,n}(\mathcal P;x)\neq0$
	for $\Rea x>-\lambda$.
	
	\noindent
\textbf{Step 1: Construction of polynomials $H_k(x)$
	satisfying $H_k(x)\neq0$ for $\Rea x>-\lambda$.}
For every integer $k\geq m+n$, define
\begin{align*}
	G_k(t,x)
	=\left(1+\frac{a_1t}{\sigma k}\right)^k
	\frac{k!}{(a_2/a_1)_k}
	L_k^{(a_2/a_1-1)}
	\left(-\frac{b_2(x+\lambda)t}{\sigma k}\right),
\end{align*}
and set
\begin{align}\label{eq:euler-operator}
	\mathcal E_N
	=\frac1{(1/\sigma)_N}
	\prod_{r=0}^{N-1}
	\left(\frac1\sigma+r+t\partial_t\right),
	\qquad F_k(t,x)=\mathcal E_NG_k(t,x).
\end{align}
When $N=0$, $\mathcal E_0$ is the identity operator.
Since
$L_k^{(a_2/a_1-1)}(0)=(a_2/a_1)_k/k!$,
we have $G_k(0,x)=1$.
Since $t\partial_t$ contributes zero to the constant term,
we use \eqref{eq:euler-operator} to obtain
\begin{align*}
	F_k(0,x)
	=\frac{G_k(0,x)}{(1/\sigma)_N}
	\prod_{r=0}^{N-1}\left(\frac{1}{\sigma}+r\right)
	=G_k(0,x)=1.
\end{align*}

We next prove that, for each fixed $x$ with $\Rea x>-\lambda$,
the polynomial $F_k(t,x)$ has degree $2k$ in $t$ and all its
zeros in a sector of opening less than $\pi/2$
inside the open left half-plane. We consider the zeros of $G_k(t,x)$ as a polynomial in $t$. 
The polynomial
$\left(1+a_1t/(\sigma k)\right)^k$
has the zero $-\sigma k/a_1$ with multiplicity $k$.
Moreover, $a_2/a_1-1>-1$, so
$L_k^{(a_2/a_1-1)}(z)$ has $k$ positive zeros.
For each such zero $\xi$, the corresponding zero of
\begin{align*}
	L_k^{(a_2/a_1-1)}
	\left(-\frac{b_2(x+\lambda)t}{\sigma k}\right)
\end{align*}
is $t=-\sigma k\xi/(b_2(x+\lambda))$.
Thus $G_k(t,x)$ has degree $2k$ in $t$, and its zeros lie
on the rays with directions $\pi$ and $\pi-\arg(x+\lambda)$.
Since $\Rea(x+\lambda)>0$, both rays lie in the open left
half-plane and enclose an angle
$|\arg(x+\lambda)|<\pi/2$.

For $0\leq r<N$, we have $1/\sigma+r>0$.
Lemma~\ref{lem:euler-sector} therefore shows that applying
each operator $1/\sigma+r+t\partial_t$ preserves degree
and keeps the zeros in this sector.
Consequently, $F_k(t,x)$ has degree $2k$ in $t$, and all its zeros lie in the sector bounded by the rays with directions $\pi$ and $\pi-\arg(x+\lambda)$.

By repeated application of the Gauss--Lucas theorem,
the zeros of $\partial_t^mF_k(t,x)$ remain in this sector.
The polynomial $\partial_t^mF_k(t,x)$ has degree
$2k-m\geq n$ in $t$.
Its value at $t=0$ is nonzero because all its zeros
have negative real parts.
For $\Rea x>-\lambda$, applying Lemma~\ref{lem:sector} to
$\partial_t^mF_k(t,x)$ yields
\begin{align}\label{eq:laguerre-approx-nonzero}
	H_k(x)
	:=(-1)^n[t^{2n}]
	\left(
	\partial_t^mF_k(t,x)
	\left.\partial_s^mF_k(s,x)\right|_{s=-t}
	\right)
	\neq0.
\end{align}
Since $F_k(t,x)$ is a polynomial in $t$ and $x$, we have
$H_k(x)\in\mathbb{C}[x]$.
	
\noindent\textbf{Step 2:
	$\mathscr T_{m,n}(\mathcal P;x)\neq0$
	for $\Rea x>-\lambda$.}
For $j\geq0$, define
$$A_{j,k}(x)=j![t^j]F_k(t,x).$$

By~\eqref{eq:euler-operator}, we have
\begin{align}\label{eq:euler-multiplier}
	\mathcal E_N(t^j)
	=\frac{\prod_{r=0}^{N-1}(1/\sigma+r+j)}{(1/\sigma)_N}\,t^j
	=\frac{(1/\sigma+j)_N}{(1/\sigma)_N}\,t^j
	=\frac{(1/\sigma+N)_j}{(1/\sigma)_j}\,t^j
	=\frac{(a_2/a_1)_j}{(1/\sigma)_j}\,t^j.
\end{align}

Next, by~\eqref{eq:generalized-laguerre}, we have 
\begin{align*}
	\frac{k!}{(a_2/a_1)_k}
	L_k^{(a_2/a_1-1)}
	\left(-\frac{b_2(x+\lambda)t}{\sigma k}\right)
	=\sum_{r=0}^k\binom{k}{r}
	\frac{1}{(a_2/a_1)_r}
	\left(\frac{b_2(x+\lambda)t}{\sigma k}\right)^r.
\end{align*}

Fix an integer $j\geq0$ and take $k\geq j$. By $F_k(t,x)=\mathcal E_NG_k(t,x)$ and~\eqref{eq:euler-multiplier}, we have 
\begin{align}\label{eq:Ajk}
	A_{j,k}(x)
	&=\frac{(a_2/a_1)_j}{(1/\sigma)_j}\,j![t^j]G_k(t,x)\notag\\
	&=\frac{(a_2/a_1)_j}{(1/\sigma)_j}\,j![t^j]\left( \left(1+\frac{a_1t}{\sigma k}\right)^k
	\frac{k!}{(a_2/a_1)_k}
	L_k^{(a_2/a_1-1)}
	\left(-\frac{b_2(x+\lambda)t}{\sigma k}\right)\right)\notag\\
	&=\frac{(a_2/a_1)_j}{(1/\sigma)_j}\,j![t^j]\left(
	\left(\sum_{s=0}^k\binom{k}{s}
	\left(\frac{a_1}{\sigma k}\right)^s t^s\right)
	\left(\sum_{r=0}^k
	\frac{\binom{k}{r}}{(a_2/a_1)_r}
	\left(\frac{b_2(x+\lambda)}{\sigma k}\right)^r t^r\right)
	\right)\notag\\
	&=\frac{(a_2/a_1)_j}{(1/\sigma)_j}\,j!
	\sum_{r=0}^j
	\binom{k}{j-r}
	\left(\frac{a_1}{\sigma k}\right)^{j-r}
	\frac{\binom{k}{r}}{(a_2/a_1)_r}
	\left(\frac{b_2(x+\lambda)}{\sigma k}\right)^r
	\notag\\
	&=\frac{(a_2/a_1)_j}{(1/\sigma)_j}\,j!
	\sum_{r=0}^j
	\frac{\binom{k}{j-r}\binom{k}{r}}{k^j(a_2/a_1)_r}
	\left(\frac{a_1}{\sigma}\right)^{j-r}
	\left(\frac{b_2}{\sigma}\right)^r(x+\lambda)^r.
\end{align}
For fixed $j$ and $0\leq r\leq j$, we obtain
\begin{align*}
	\lim_{k\to\infty}
	\frac{\binom{k}{j-r}\binom{k}{r}}{k^j}
	=\frac{1}{(j-r)!r!}.
\end{align*}
It follows from~\eqref{eq:Ajk} that
\begin{align}\label{eq:Ajklimit}
	\lim_{k\to\infty}A_{j,k}(x)
	&=\left(\frac{a_1}{\sigma}\right)^j
	\frac{(a_2/a_1)_j}{(1/\sigma)_j}
	\sum_{r=0}^j\binom{j}{r}
	\frac{(b_2(x+\lambda)/a_1)^r}{(a_2/a_1)_r}
	=P_j(x),
\end{align}
where the last equality follows from
  \eqref{eq:generalized-laguerre} and \eqref{eq:laguerre-branch-explicit}.
The finite sum in \eqref{eq:Ajk} also shows that the convergence
is coefficientwise.
Since $\deg A_{j,k}\leq j$, the convergence is uniform
on every compact subset of $\mathbb C$.
	
We now use \eqref{eq:Ajklimit} to determine
the limit of $H_k(x)$. 
By $A_{j,k}(x)=j![t^j]F_k(t,x)$ and \eqref{eq:laguerre-approx-nonzero}, we have 
\begin{align*}
	H_k(x)
	&=(-1)^n[t^{2n}]
	\left(
	\partial_t^mF_k(t,x)
	\left.\partial_s^mF_k(s,x)\right|_{s=-t}
	\right)\\
	&=(-1)^n[t^{2n}]
	\left(
	\left(\sum_{r=0}^{2k-m}A_{m+r,k}(x)\frac{t^r}{r!}\right)
	\left(\sum_{r=0}^{2k-m}A_{m+r,k}(x)\frac{(-t)^r}{r!}\right)
	\right)\\
	&=\sum_{j=0}^{2n}\frac{(-1)^{n+j}}{j!(2n-j)!}
	A_{m+j,k}(x)A_{m+2n-j,k}(x).
\end{align*}

It follows from \eqref{eq:Ajklimit} and \eqref{eq:turan} that
\begin{align*}
\lim_{k\to\infty}H_k(x)
	=\mathscr T_{m,n}(\mathcal P;x),
\end{align*}
with uniform convergence on compact subsets of $\mathbb C$.
Lemma~\ref{lem:nonzero} shows that the limit is not identically zero.
By \eqref{eq:laguerre-approx-nonzero}, every $H_k$ is nonvanishing
in the half-plane $\Rea x>-\lambda$. Therefore, Hurwitz theorem
yields $\mathscr T_{m,n}(\mathcal P;x)\neq0$ for $\Rea x>-\lambda$.
Since $\lambda\geq0$, weak Hurwitz stability follows. This completes the proof.
\end{proof}





\section*{Declaration of Generative AI}
The author used ChatGPT to assist with research and writing. The author verified all arguments and results and takes full responsibility for the manuscript.

\end{document}